\documentclass[12pt]{amsart}

\calclayout
\usepackage[vcentermath]{youngtab}
\usepackage{pgf}
\usepackage{graphicx}
\usepackage{amsmath, amsthm, amssymb}

\usepackage{young}

\newcommand\Z{{\mathbb Z}}
\newcommand\N{{\mathbb N}}

\newcommand\C{{\mathbb C}}

\newcommand\be{{\bf 1}}

\newcommand\Cc{{\mathcal C}}
\newcommand\Dd{{\mathcal D}}

\newcommand\ssl{\mathfrak{sl}}
\newcommand\sso{\mathfrak{so}}
\newcommand\ssp{\mathfrak{sp}}

\newcommand\asl{\widehat{\mathfrak{sl}}}

\newcommand\aLam{\widehat{\Lambda}}
\newcommand\alam{\widehat{\lambda}}
\newcommand\flam{\underline{\lambda}}
\newcommand\fLam{\underline{\Lambda}}

\newcommand\alamt{\widehat{\lambda^t}}

\newtheorem{thm}{Theorem}[section]
\newtheorem*{theorem}{Theorem}
\newtheorem{prop}[thm]{Proposition}
\newtheorem{lem}[thm]{Lemma}
\newtheorem{cor}[thm]{Corollary}

\theoremstyle{remark}
\newtheorem{rem}[thm]{Remark}

\theoremstyle{definition}

\newtheorem{eg}[thm]{Example}

\author{Victor Ostrik}
\address{Department of Mathematics, University of Oregon, Eugene, OR 97403-1222, USA}
\address{Laboratory of Algebraic Geometry,
National Research University Higher School of Economics, Moscow, Russia}
\author{Eric C. Rowell}
\address{ Department of Mathematics, Texas A\&M University, College Station, TX 77843-3368}
\author{Michael Sun}
\begin{document}
\email{vostrik@uoregon.edu, rowell@math.tamu.edu, michaelysun@outlook.com}
\title{Symplectic level-rank duality via tensor categories}

\begin{abstract} We give two proofs of a level-rank duality for braided fusion categories obtained from quantum groups of type $C$ at roots of unity.  The first proof uses conformal embeddings, while the second uses a classification of braided fusion categories associated with quantum groups of type $C$ at roots of unity. In addition we give a similar result for non-unitary braided fusion categories quantum groups of types $B$ and $C$ at odd roots of unity.
\end{abstract}
\thanks{VO partially supported by the NSF grant DMS-1702251 and by the Russian Academic Excellence Project '5-100'. ECR is partially supported by US NSF grant DMS-1664359, a Presidential Impact Fellowship of Texas A\&M, and a Simons Fellowship. This paper is based upon work supported by the National Science Foundation under the Grant No. DMS-1440140 while the first and  and second authors were in residence at the Mathematical Sciences Research Institute during the Spring 2020 semester.
MS thanks the Research Center for Operator Algebras at East China Normal University for supporting me, as well as the support of Shane Bastin, Kathy Kong and NEIS.
}

\maketitle

\section{Introduction}
There are a number of results which connect classical affine Lie algebras with interchanged level and rank; such phenomena are generally called {\em level-rank} duality. One manifestation of this is that the braided fusion categories associated with such Lie algebras are closely related.  
In type $A$ a result of this type was proved in \cite{OS}. It states that there is a braid-reversing tensor equivalence between $\Cc(\ssl_n)_k^0$ and $\Cc(\ssl_k)_n^0$ where $\Cc({\ssl}_n)_k^0$
is the adjoint subcategory of the modular tensor category $\Cc(\ssl_n)_k$ obtained from the affine Lie algebra $\widehat\ssl_n$ at level $k$. Recall that an alternative construction of $\Cc(\ssl_n)_k$ is via the semisimplification of the category of tilting modules of the quantum group $U_q\ssl_n$ specialized at $q=e^{\pi i/(n+k)}$, typically denoted $\Cc(\ssl_n,n+k)$.

In this paper we examine the analogous situation for the type $C$ case associated with symplectic Lie algebras $\ssp_{2n}$.  Just as in the type $A$ case, the unitary modular category $\Cc(\ssp_{2n})_k$  may be constructed in (at least) two ways: 1) as the semisimplification $\Cc(\ssp_{2n},2k+2n+2)$ of the subcategory of tilting modules in $Rep(U_q \ssp_{2n})$ for the choice $q=e^{\pi i/(2k+2n+2)}$ and 2) as the category $\Cc(\ssp_{2n})_k$ of level $k$ representations of the affine Lie algebra $\widehat{\ssp}_{2n}$ under the level preserving tensor product (see \cite{BK}). The braided monoidal equivalence of the categories constructed from these two approaches can be found in \cite{finkelberg,Huang}.  Despite this equivalence we will use both notations to distinguish the approaches.

From any braided fusion category $\Cc$ we may obtain a new braided fusion category $\Cc^{rev}$ with the same underlying fusion category by replacing the braiding isomorphisms $c_{X,Y}$ with $c_{Y,X}^{-1}$.
For any $\Z/2$-graded braided fusion category $\Cc=\Cc_0\oplus\Cc_1$ we may also obtain a new braided structure by replacing the braiding $c_{X,Y}$ for $X,Y\in\Cc_1$  by $-c_{X,Y}$, leaving the braiding unchanged if either of $X$ or $Y$ are in $\Cc_0$.  One way to construct this new braided category directly is to take the diagonal of the $\Z/2\times\Z/2$-graded category $\Cc\boxtimes sVec$.   We will denote by $\Cc^{-}$ the braided fusion category obtained in this way.
We prove the following:
\begin{theorem} There is a braid-reversing equivalence between
$\Cc(\ssp_{2n})_k$ and $\Cc(\ssp_{2k})_n^{-}$.
\end{theorem}
We provide two proofs of this theorem.  The first, direct, proof follows the strategy of \cite{OS} using the conformal embedding $(\widehat{\mathfrak{sp}}_{2n})_k\oplus (\widehat{\mathfrak{sp}}_{2k})_n\subset (\widehat{\mathfrak{so}}_{4nk})_1.$  The second employs reconstruction techniques for braided fusion categories with fusion rules of type $C$, found in \cite{TW}.  While the latter proof is shorter, it invokes some heavy categorical machinery.  On the other hand, this categorical approach allows us to prove a related result for $\Cc(\ssp_{2n},\ell)$ for \emph{odd} $\ell$ where now the category $\Cc(\sso_{2k},\ell)$ with $\ell=2k+2n+1$ plays the role of the dual.  These are typically non-unitarizable categories, and hence cannot be constructed from affine Lie algebras in the standard way.  

It would be interesting to extend our results to the orthogonal case where the level-rank duality
connects affine Lie algebras $(\widehat{\mathfrak{so}}_{n})_k$ and $(\widehat{\mathfrak{so}}_{k})_n$. However this case seems to be technically more involved than the symplectic case 
and it is not considered in this paper. We refer the reader to \cite{Has,JL,Muk} for some interesting results in this direction.

This article is organized as follows.  In Section \ref{prelim} we lay the basic Lie theoretic and combinatorial groundwork and in Section \ref{branching} we describe the key conformal embedding.  Section \ref{main results} contains two proofs of our level-rank duality theorem.  The appendix contains a detailed proof of the Kac-Peterson formula in the symplectic case.

\section{Preliminaries}\label{prelim}
\subsection{Combinatorics}\label{combinatorics}
Let $\lambda = (\lambda_1\geq\dots\geq\lambda_n\ge 0)$ be a \emph{partition} of $|\lambda | := \lambda_1 + \dots + \lambda_n$ with $n$ parts (see e.g. \cite{Mac}). We will identify $\lambda$ with its corresponding \emph{Young diagram}. Let $I_{n,k}$ be the set of all partitions $\lambda$ whose Young diagram fits into an $n\times k$ rectangle; in other words $I_{n,k}$ consists of partitions with $n$ parts and $\lambda_1\le k$. We will denote by $I_{n,k}^0$ (respectively 
$I_{n,k}^1$) the subset of $I_{n,k}$ consisting of partitions $\lambda$ such that $|\lambda|$ is
even (respectively odd).

Denote by $\lambda^t$ the transposed partition of $\lambda$. Clearly, $\lambda\in I_{n,k}$ implies $\lambda^t\in I_{k,n}$. Denote by $\lambda^c$ the transpose of the complement of $\lambda$ in an $n\times k$ rectangle. Again,  $\lambda\in I_{n,k}$ implies $\lambda^c\in I_{k,n}$. We will also
consider a composition $(\lambda^t)^c=(\lambda^c)^t=:\lambda^{tc}$ which preserves the set
$I_{n,k}$.

\begin{eg} 
Let $\lambda =(2,1,1)\in I_{3,2}$. Then $\lambda^t=(3,1)\in I_{2,3}$, $\lambda^c=(2,0)\in I_{2,3}$,
and $\lambda^{tc}=(1,1,0)\in I_{3,2}$.
\end{eg}

Let $C_{n,k} :=\{(k_0,k_1,..., k_n) \in \N^{n+1} | \sum_i k_i = k\}$ be the set of  dominant weights for $\widehat{\mathfrak{sp}}_{2n}$
of level $k$. We identify the sets $C_{n,k}$ and $I_{n,k}$ via the mutually inverse bijections
$c_{n,k}: I_{n,k}\to C_{n,k}$, 
$c_{n,k}(\lambda) := (k-\lambda_1,\lambda_1- \lambda_2,\lambda_2-\lambda_3,\dots,\lambda_{n-1}-\lambda_n, \lambda_n)$ and $d_{n,k}: C_{n,k}\to I_{n,k}$, $d_{n,k}(k_0, k_1, \dots, k_{n})=(k_1+\dots+k_{n},k_2+\cdots+k_{n},\dots, k_{n})$. With these identifications the bijection $C_{n,k}\to C_{k,n}$
corresponding to $\lambda \mapsto \lambda^c$ can be described as follows (see \cite{KP}):
if we consider $n+k$ dots, first represent $(k_0,\dots, k_n)\in C_{n,k}$ by picking $n$ dots to be black (leaving the other $k$ dots white) so that there are $k_0$ white dots before the first black dot, $k_1$ white dots between the first and second black dot, and continue this until the last $k_n$ dots are after the last black dot. In this way the black dots partition the white dots according to an element of $C_{n,k}$. Its corresponding element in $C_{k,n}$ is how the white dots partition the black dots (see example below). 
\begin{eg}\label{bwdots} Let $n=7$, $k=6$ consider $n+k=13$ dots
$$\bullet\bullet\circ\bullet\bullet\bullet\bullet\circ\circ\circ\bullet\circ \circ$$
$$C_{7,6}\ni (0,0,1,0,0,0,3,2) \mapsto(2,4,0,0,1,0,0)\in C_{6,7}.$$
The corresponding partitions are $(6,6,5,5,5,5,2)\in I_{7,6}$ and $(5,1,1,1)\in I_{6,7}$.
\end{eg}

The bijection $\lambda \mapsto \lambda^{tc}$ also has a simple interpretation in this language:
one has to read the diagram of black and white dots representing $\lambda$ backwards.
\subsection{Symplectic Lie algebra} \label{symplectic}

Let $\mathfrak{sp}_{2N}$ be the Lie algebra of $2N\times 2N$ symplectic matrices over $\C$,
see e.g. \cite{Hum}.
\subsubsection{The root system of $\mathfrak{sp}_{2N}$}
Following \cite{Hum}, let $e_1,e_2,\dots,e_N$ be an orthonormal basis for $\mathfrak{h}^*$ such that the simple roots are given by
$$e_1-e_2, e_2-e_3,\dots, e_{N-2}-e_{N-1}, e_{N-1}-e_N,2e_N.$$
The positive long roots are therefore given by
$$2e_1,2e_2,\dots, 2e_N$$
Let $W$ be the Weyl group of $\mathfrak{sp}_{2N}$. We recall that the group $W$ identifies
with the group of signed permutations of the basis $e_1,e_2,\dots,e_N$, see \cite{Hum}.

\subsubsection{Affinization} Let $\widehat{\mathfrak{sp}}_{2N}$ be the affinization of $\mathfrak{sp}_{2N}$ (see e.g. \cite[Chapter 7]{Kac}). 
Thus $\widehat{\mathfrak{sp}}_{2N} = (\mathfrak{sp}_{2N}\otimes\C[t,t^{-1}])\oplus\C K$ where the element $K$ is central and (for $\kappa$ the Killing form)
$$[x\otimes t^a,y\otimes t^b]=[x,y]\otimes t^{a+b}+a\kappa(x,y)\delta_{a,-b}K.$$
\subsubsection{The root system of $\widehat{\mathfrak{sp}}_{2N}$}\label{affroots}
Let $\hat{W}$ be the affine Weyl group. Let $T\subset \hat{W}$ be the subgroup of translations, so $\hat{W}=W\rtimes T$ (\cite[Proposition 6.5]{Kac}). Write $\hat{\Delta}$ and $\hat{\Delta}^+$ for the affine root system and its positive roots respectively. Let $\delta$ be the indivisible positive imaginary root. Then $w(\delta)=\delta$ for all $w\in\hat{W}$ and the imaginary roots are nonzero multiples of $\delta$. The real roots of $\hat{\Delta}$ are $\Delta+\Z\delta$, with $\hat{\Delta}^+=\Delta^+\cup (\Delta+\Z_{>0}\delta)$. 
\subsubsection{Representations}Let $\Cc(\mathfrak{sp}_{2n})_k$ be the category of highest weight integrable $\widehat{\mathfrak{sp}}_{2n}$ modules of level $k$, see e.g. \cite[Chapter 10]{Kac}.
 The simple objects of $\Cc(\mathfrak{sp}_{2n})_k$ are labeled by partitions $\lambda \in I_{n,k}$. Write $\hat{\lambda}$ for the simple $\Cc(\mathfrak{sp}_{2n})_k$ module of highest weight $c_{n,k}(\lambda)$. 
We will also denote by $\underline{\lambda}$ the corresponding irreducible finite dimensional
module over $\mathfrak{sp}_{2n}$; thus the highest weight of $\underline{\lambda}$ is
$(k_1,\ldots,k_n)$ if $c_{n,k}(\lambda)=(k_0,k_1,\ldots,k_n)$.

\subsection{Orthogonal Lie algebra} \label{ortho}
Let $\mathfrak{so}_{N}$ be the Lie algebra of form preserving endomorphisms on an $N$-dimensional vector space equipped with a symmetric non-degenerate bilinear form and let $\widehat{\mathfrak{so}}_{N}$ be its affinization. We will use the same
notations as in Section \ref{symplectic} to discuss the  highest weight integrable representations
of $\widehat{\mathfrak{so}}_{N}$.
Recall that in the case when $N$ is divisible by 4 the simple objects of $\Cc(\mathfrak{so}_{N})_1$ are $\aLam_0, \aLam_1, \aLam_+, \aLam_-$
where $\fLam_0$ is the trivial $\mathfrak{so}_{N}-$module, $\fLam_1$ is the natural $\mathfrak{so}_{4nk}-$module
of dimension $N$, and $\fLam_\pm$ are two half-spinor $\mathfrak{so}_{N}-$modules. 

Now assume that $N=4nk$ and that the symmetric bilinear form on $\C^N$ is the tensor product
of the symplectic forms on the spaces $\C^{2n}$ and $\C^{2k}$. Thus we have an embedding 
$\mathfrak{sp}_{2n}\oplus \mathfrak{sp}_{2k}\subset \mathfrak{so}_{4nk}$ and we can
distinguish
between $\fLam_+$ and $\fLam_-$ in the following way: $\fLam_+$ has some nonzero
$\mathfrak{sp}_{2n}-$invariant vector while $\fLam_-$ has no nonzero $\mathfrak{sp}_{2n}-$invariant vectors, see
Remark \ref{ccfinite} below.

{\bf Warning:} Notation convention for $\aLam_\pm$ may change if we interchange the roles
of $\mathfrak{sp}_{2n}$ and $\mathfrak{sp}_{2k}$, see Remark \ref{ccfinite}.

\subsection{Conformal embeddings and Kac-Peterson formula}
It is a well known fact and a consequence of general results in \cite{KP} that the embedding
$\mathfrak{sp}_{2n}\oplus \mathfrak{sp}_{2k}\subset \mathfrak{so}_{4nk}$ induces 
a {\em conformal embedding}
\begin{equation}\label{tg}
(\widehat{\mathfrak{sp}}_{2n})_k\oplus (\widehat{\mathfrak{sp}}_{2k})_n\subset (\widehat{\mathfrak{so}}_{4nk})_1.
\end{equation}
This means that a restriction of a highest weight integrable $\widehat{\mathfrak{so}}_{4nk}-$module 
of level 1 to $\widehat{\mathfrak{sp}}_{2n}\oplus \widehat{\mathfrak{sp}}_{2k}$ is again a highest weight integrable module on which the central elements of $\widehat{\mathfrak{sp}}_{2n}$ and $\widehat{\mathfrak{sp}}_{2k}$ act as $k\mbox{Id}$ and $n\mbox{Id}$ respectively.

The following result is \cite[Proposition 2]{KP}. Since the proof is omitted in {\em loc. cit.} 
we provide its derivation from \cite[Proposition 1]{KP} in  the appendix (Section \ref{KPproof}).

\begin{prop}\label{KPcc}
There is an isomorphism of 
$(\widehat{\mathfrak{sp}}_{2n})_k\oplus (\widehat{\mathfrak{sp}}_{2k})_n-$modules:
$$\aLam_+\oplus \aLam_-\cong \bigoplus_{\lambda \in I_{n,k}}
\alam \boxtimes \widehat{\lambda^c}.$$
\end{prop}

\subsection{Tensor categories}\label{tcat}
Recall that 
a monoidal category $(\Cc,\otimes)$ is \emph{braided} if there is a natural bifunctor
isomorphism $c_{X,Y} : X \otimes Y \to Y \otimes X$ called a braiding subject to the
hexagon axioms (see e.g \cite[Definition 8.1.1]{egnobook}). For every braided tensor category $(\Cc, \otimes, c)$, there is a
reversed braiding on $\Cc$ given by $c_{X,Y}^{rev} = c_{Y,X}^{-1}$. A braided tensor category
$\Cc$ endowed with the reversed braiding will be denoted $\Cc^{rev}$. A monoidal functor (a functor $T : \Cc \to \Dd$, together with a natural isormorphism $J:T(-\otimes-)\to T(-)\otimes T(-)$) is said to be a braided equivalence of categories if $c_{T(X),T(Y)}J_{X,Y}=J_{Y,X}T(c_{X,Y})$ and $T$ is an equivalence of the underlying categories (see e.g. \cite[Definition 8.1.7]{egnobook}). A \emph{braid--reversing} equivalence of $\Cc$ and $\Dd$ is a braided equivalence $\Cc \simeq \Dd^{rev}$. Two objects $X,Y$ in a braided tensor category $\Cc$ are said to mutually centralize each other, in the sense of \cite{Muger}, if $c_{X,Y}c_{Y,X}=\mbox{id}_{X\otimes Y}$. 

It is a deep and important fact (see e.g. \cite[Chapter 7]{BK}) that the categories
$\Cc(\mathfrak{sp}_{2n})_k$ and $\Cc(\mathfrak{so}_{4nk})_1$ have a natural structure 
of {\em modular tensor categories}. In particular they are braided rigid monoidal categories
which are {\em non-degenerate} in the sense of \cite[8.19]{egnobook}. We will often use
the following special property: in the categories $\Cc(\mathfrak{sp}_{2n})_k$ and $\Cc(\mathfrak{so}_{4nk})_1$ every object is self-dual.

It is well known that the category $\Cc(\mathfrak{sp}_{2n})_k$ is $\Z/2-$graded (see \cite[4.14]{egnobook});
the grading is given by the weights of an object modulo the root lattice. Thus the parity of the object $\alam$ is $|\lambda|\pmod{2}$.

Similarly the category $\Cc(\mathfrak{so}_{4nk})_1$ is graded by the quotient of the weight lattice
of $\mathfrak{so}_{4nk}$ modulo the root lattice; this group is known to be $\Z/2\oplus \Z/2$.
This implies that the category $\Cc(\mathfrak{so}_{4nk})_1$ is {\em pointed} (i.e. all simple objects are invertible) with group of simple objects isomorphic to
$\Z/2\oplus \Z/2$. The unit object of the category $\Cc(\mathfrak{so}_{4nk})_1$ is $\aLam_0$. It is a well known and easy to check fact that the object $\aLam_1\in \Cc(\mathfrak{so}_{4nk})_1$ is a {\em fermion}. This means that the braiding $c_{\aLam_1,\aLam_1}$ equals $-\mbox{id}_{\aLam_1,\aLam_1}$ or, equivalently, that the
subcategory of $\Cc(\mathfrak{so}_{4nk})_1$ generated by $\aLam_1$ is equivalent to the category of super vector spaces as a braided fusion category. This gives a simple construction of 
the category $\Cc^-$ which appears in the Introduction: let $\Cc=\Cc^0\oplus \Cc^1$ be a 
$\Z/2-$graded braided fusion category. Consider the subcategory $\Cc_0=\Cc^0\boxtimes \aLam_0\oplus \Cc^1\boxtimes \aLam_1$ of $\Cc \boxtimes \Cc(\mathfrak{so}_{4nk})_1$. It
is clear that $\Cc_0$ is closed under the tensor product; moreover we have a braided equivalence
$\Cc^-\simeq \Cc_0$ which sends $X\in \Cc_0$ to $X\boxtimes \aLam_0$ and $X\in \Cc_1$ to
$X\boxtimes \aLam_1$.

\subsection{\'Etale algebras from conformal embeddings}
An \'etale algebra in a semisimple braided tensor category $\Cc$ is defined to be an object
$A \in \Cc$ endowed with an associative commutative unital multiplication
and that the category $\Cc_A$ of right $A$-modules is semisimple, see \cite[
Definition 3.1]{KO2}. An \'etale algebra $A$ is called connected if the unit object
appears in $A$ with multiplicity 1. For a connected \'etale algebra $A \in\Cc$
the category $\Cc_A$ with operation $\otimes_A$ of tensor product over $A$ is naturally
a fusion category, see e.g. \cite[Section 3.3]{KO2}. The category $\Cc_A$ contains
a full tensor Serre subcategory $\Cc_A^{dys}$ of dyslectic modules, which is also
naturally braided. See for example \cite[Section 3.5]{KO2}. 

A general result observed in \cite[Theorem 5.2]{KO} states that for any conformal embedding the pullback of the vacuum module is an \'etale algebra (see \cite{HKL,CKM2} for a proof); moreover taking pullbacks is a braided equivalence with the category of dyslectic modules over this algebra, see \cite{CKM} for a proof. Specializing this to the conformal embedding (\ref{tg}) we get
 
 \begin{thm}\label{etale}
 Let $A$ be the restriction of $\hat{\Lambda}_0$ under the embedding (\ref{tg}).
Then A is a connected \'etale algebra in $ \Cc(\mathfrak{sp}_{2n})_k\boxtimes \Cc(\mathfrak{sp}_{2k})_n$. Moreover,
the restriction functor is a braided equivalence $\Cc(\mathfrak{so}_{4nk})_1\simeq  (\Cc(\mathfrak{sp}_{2n})_k\boxtimes \Cc(\mathfrak{sp}_{2k})_n)_A^{dys}$. Hence the pullbacks of $\aLam_1$, $\aLam_+$ and $\aLam_-$ are precisely the simple dyslectic $A$-modules in $\Cc(\mathfrak{sp}_{2n})_k\boxtimes \Cc(\mathfrak{sp}_{2k})_n$. 
\end{thm}

\subsection{Lagrangian algebras in products} We recall that a connected \'etale algebra $A$ in a non-degenerate braided fusion category $\Cc$ is {\em Lagrangian} if the category $\Cc_A^{dys}$ is
trivial or, equivalently, $\mbox{FPdim}(A)^2=\mbox{FPdim}(\Cc)$, see \cite[Definition 4.6]{KO2}. 

The following result is a version of \cite[Theorem 3.6]{DNO}.

\begin{thm} \label{equivalences}
Let $\Cc$ and $\Dd$ be two non-degenerate braided fusion categories and let $A\in \Cc \boxtimes \Dd$ be a Lagrangian algebra. Let $A_1=A\cap (\Cc \boxtimes {\bf 1})$ and let
$A_2=A\cap ({\bf 1}\boxtimes \Dd)$. Then $A_1\in \Cc$ and $A_2\in \Dd$ are connected \'etale algebras
and there exists a braid reversing equivalence $\phi: \Cc_{A_1}^{dys}\simeq \Dd_{A_2}^{dys}$
such that $A=\bigoplus_{M\in {\mbox{Irr}(\Cc_{A_1}^{dys})}}M^*\boxtimes \phi(M)$ as an object of
$\Cc \boxtimes \Dd$ (the summation is over the isomorphism classes of simple objects of
$\Cc_{A_1}^{dys}$).

Conversely, given two connected \'etale algebras $A_1\in \Cc$ and $A_2\in \Dd$ and 
a braid reversing equivalence $\phi: \Cc_{A_1}^{dys}\simeq \Dd_{A_2}^{dys}$, the object
$A=\bigoplus_{M\in {\mbox{Irr}(\Cc_{A_1}^{dys})}}M^*\boxtimes \phi(M)\in \Cc \boxtimes \Dd$
has the structure of a Lagrangian algebra.
\end{thm}

\begin{eg}\label{good equiv} 
Assume that $A\in \Cc \boxtimes \Dd$ is a Lagrangian algebra such that 
$A_1=A\cap (\Cc \boxtimes {\bf 1})={\bf 1}\boxtimes {\bf 1}$. Then Theorem \ref{equivalences}
says that there exists a braid reversing equivalence $\phi :\Cc \simeq \Dd_{A_2}^{dys}$.
If we write $A=\bigoplus_{X_i\in {\mbox{Irr}(\Cc)}}X_i^*\boxtimes Y_i$ then the functor
$\phi$ sends $X_i$ to $A_2-$module which is $Y_i$ as an object of $\Dd$.
\end{eg} 

\section{Branching rules}\label{branching}

In this Section we will present the branching rules for conformal embedding \eqref{tg}.

The unit object $\mathbf{1}_{n,k}$ of $\Cc(\mathfrak{sp}_{2n})_k$ corresponds to the empty Young diagram in $I_{n,k}$ and $\mathbf{1}_{n,k}^c$ is a unique nontrivial invertible object corresponding in $I_{k,n}$ to the unique diagram with $kn$ boxes. For $\lambda\in I_{n,k}$, tensoring in $\Cc(\mathfrak{sp}_{2k})_n$ is given by
$$\mathbf{1}_{n,k}^c \otimes \alam \cong \widehat{\lambda}^{tc},$$
see \cite{Fuchs}. We will just write $\mathbf1$ and $\mathbf1^c$ for convenience.

\begin{lem} \label{invcc}
The $A-$modules $(\mathbf{1}\boxtimes\mathbf{1}^c)\otimes A$ and $(\mathbf{1}^c \boxtimes\mathbf1)\otimes A$
are simple and dyslectic. Moreover there exists a sign $s=\pm$ such that
$(\mathbf{1}\boxtimes\mathbf1^c)\otimes A\cong \aLam_s$.
\end{lem} 
\begin{proof}
By Proposition \ref{KPcc},
$\mathbf{1}\boxtimes \mathbf{1}^c$ and $\mathbf{1}^c \boxtimes \be$ are direct summands of
$\aLam_+\oplus \aLam_-$.  Hence the result follows from Theorem \ref{etale}. 
\end{proof}

{\bf Warning:} We do not claim that these two modules are (or are not) isomorphic. We will
see that this depends on the values of $n$ and $k$.

Recall that the category $\Cc(\mathfrak{sp}_{2n})_k$ is $\Z/2-$graded. We denote
by $\Cc(\mathfrak{sp}_{2n})_k^0$ and $\Cc(\mathfrak{sp}_{2n})_k^1$ its trivial and nontrivial 
components. Recall that $\Cc(\mathfrak{sp}_{2n})_k^0$ coincides with the centralizer of the object
$\mathbf1^c$.

\begin{cor}\label{paircc} (i) The object $\aLam_0=A$ is contained in $\Cc(\mathfrak{sp}_{2n})_k^0\boxtimes \Cc(\mathfrak{sp}_{2k})_n^0$.

(ii) The object $\aLam_1$ is contained in $\Cc(\mathfrak{sp}_{2n})_k^1\boxtimes \Cc(\mathfrak{sp}_{2k})_n^1$.
\end{cor}

\begin{proof} By Lemma \ref{invcc} and \cite[Lemma 3.15]{DGNO} $A$ centralizes $\mathbf{1}\boxtimes\be^c$ and $\be^c \boxtimes \be$. This implies (i). It is clear that
$\aLam_1$ contains $\alam_1\boxtimes\alam_1$ where $\alam_1$ is the one box Young diagram;
hence $\aLam_1$ is a direct summand of $(\alam_1\boxtimes\alam_1)\otimes A$. Thus (i) implies (ii).
\end{proof}

We arrive at the main result of this section:
\begin{thm}[cf \cite{Has}]\label{maincc}
There are isomorphisms of 
$(\widehat{\mathfrak{sp}}_{2n})_k\oplus (\widehat{\mathfrak{sp}}_{2k})_n-$modules:
$$\aLam_0\cong \bigoplus_{\lambda \in I_{n,k}^0}\alam \boxtimes \alamt ,$$
$$\aLam_1\cong \bigoplus_{\lambda \in I_{n,k}^1}\alam \boxtimes \alamt ,$$
$$\aLam_+\cong \bigoplus_{\lambda \in I_{n,k}^0}\alam \boxtimes \widehat{\lambda^c},$$
$$\aLam_-\cong \bigoplus_{\lambda \in I_{n,k}^1}\alam \boxtimes \widehat{\lambda^c}.$$

\end{thm}
\begin{proof}
Let $s=\pm$ be as in Lemma \ref{invcc}.
The multiplication rules in the group $\Z/2\oplus \Z/2$ imply that 
$$\aLam_s\otimes (\aLam_+\oplus \aLam_-)\cong \aLam_0\oplus \aLam_1,$$
where the tensor product is taken in the category $\Cc(\mathfrak{so}_{4nk})_1$. Using Theorem \ref{etale},
the definition of $s$, and Proposition \ref{KPcc} we have equivalently
$$\begin{aligned}\aLam_0\oplus \aLam_1&\cong \aLam_s\otimes_A (\aLam_+\oplus \aLam_-)\\
&\cong((\mathbf{1}\boxtimes{\mathbf1}^c)\otimes A)\otimes_A\bigoplus_{\lambda \in I_{n,k}}(\alam \boxtimes\alam^c)\\
&\cong(\mathbf{1}\boxtimes{\mathbf1}^c)\otimes\bigoplus_{\lambda \in I_{n,k}}(\alam \boxtimes\alam^c)\\
&\cong\bigoplus_{\lambda \in I_{n,k}}\alam \boxtimes(\mathbf{1}^c \otimes \alam^c)\\
&\cong\bigoplus_{\lambda \in I_{n,k}}\alam \boxtimes\alamt,\end{aligned}$$
where the tensor product in the third and fourth line is taken in $\Cc(\mathfrak{sp}_{2n})_k\boxtimes\Cc(\mathfrak{sp}_{2k})_n$ and in 
$\Cc(\mathfrak{sp}_{2k})_n$ respectively. Using Corollary \ref{paircc} we get the required decompositions of $\aLam_0$ and
$\aLam_1$. Now by definition of $s$ we get
$\aLam_s=(\mathbf{1}\boxtimes\mathbf{1}^c)\otimes A\cong \bigoplus_{\lambda \in I_{n,k}^0}\alam \boxtimes \widehat{\lambda^c}$ and therefore $\aLam_{-s}=\bigoplus_{\lambda \in I_{n,k}^1}\alam \boxtimes \widehat{\lambda^c}$. Then by Remark \ref{ccfinite} we see that $s=+$
and the result follows.
\end{proof}
\begin{rem} \label{ccfinite} 
(i) We see that $\mathbf{1}^c \boxtimes \be$ appears
in the decomposition of the same $\aLam_+$ if and only if $nk$ is even. Thus the modules
$(\mathbf{1}\boxtimes\mathbf{1}^c)\otimes A$ and $(\mathbf{1}^c \boxtimes\mathbf1)\otimes A$
from Lemma \ref{invcc} are isomorphic if and only if $nk$ is even.

(ii) It was observed by Hasegawa in \cite{Has} that all summands in the decompositions of $\aLam_\pm$ have the same {\em conformal dimension}. It follows that we have similar decompositions
$$\fLam_{\, +}\cong \bigoplus_{\lambda \in I_{n,k}^0}\flam \boxtimes \underline{\lambda^c},\;\;\;\;
\fLam_{\, -}\cong \bigoplus_{\lambda \in I_{n,k}^1}\flam \boxtimes \underline{\lambda^c}$$
for the branching under the finite dimensional algebras embedding $\mathfrak{sp}_{2n}\oplus \mathfrak{sp}_{2k}\subset \mathfrak{so}_{4nk}$ (see Proposition \ref{affine} for a closely related statement). In particular the summand of $\fLam_{\, +}$
corresponding to the empty Young diagram contains vectors invariant under the action of
$\mathfrak{sp}_{2n}$ which justifies our choice for labeling of $\fLam_\pm$ in Section \ref{ortho}.
Note that the conformal dimensions of the summands of $\aLam_0$ and $\aLam_1$ are not constant.
\end{rem}

\section{Level rank duality}\label{main results}
In this section we prove the main result of this paper. Recall the construction of the category
$\mathcal{C}(\mathfrak{sp}_{2n})_k^-$ from the introduction.

\begin{thm}\label{main theorem} There is a braid-reversing monoidal equivalence $$\mathcal{C}(\mathfrak{sp}_{2n})_k\simeq\mathcal{C}(\mathfrak{sp}_{2k})_n^{-}$$
sending
$\widehat{\lambda}\mapsto\widehat{\lambda^t}$.
In particular, $\mathcal{C}(\mathfrak{sp}_{2n})_k^0$ and $\mathcal{C}(\mathfrak{sp}_{2k})_n^0$ are braid reversing equivalent.
\end{thm}
\subsection{First proof}
\begin{proof} By Theorems \ref{etale} and \ref{maincc} we see that an object $$A=\bigoplus_{\lambda \in I_{n,k}^0}\alam \boxtimes \alamt \in \Cc(\mathfrak{sp}_{2n})_k\boxtimes \Cc(\mathfrak{sp}_{2k})_n$$ has the structure of a connected \'etale algebra 
such that we have a braided equivalence $(\Cc(\mathfrak{sp}_{2n})_k\boxtimes \Cc(\mathfrak{sp}_{2k})_n)_A^{dys}\simeq \Cc(\mathfrak{so}_{4nk})_1$. Thus
by Theorem \ref{equivalences} we get a Lagrangian algebra
 $$\tilde{A}\in \mathcal{C}(\mathfrak{sp}_{2n})_k\boxtimes\mathcal{C}(\mathfrak{sp}_{2k})_n\boxtimes\mathcal{C}(\mathfrak{so}_{4nk})_1^{rev}$$
  given by
$$\left(\bigoplus_{\lambda\in I_{n,k}^0}\hat{\lambda}\boxtimes\hat{\lambda^t}\boxtimes\hat{\Lambda}_0\right)\oplus\left(\bigoplus_{\lambda\in I_{n,k}^1}\hat{\lambda}\boxtimes\hat{\lambda^t}\boxtimes\hat{\Lambda}_1\right)$$
$$\oplus\left(\bigoplus_{\lambda\in I_{n,k}^0}\hat{\lambda}\boxtimes\hat{\lambda^c}\boxtimes\hat{\Lambda}_+\right)\oplus\left(\bigoplus_{\lambda\in I_{n,k}^1}\hat{\lambda}\boxtimes\hat{\lambda^c}\boxtimes\hat{\Lambda}_{-}\right).$$

Let us use Theorem \ref{equivalences} again with $\Cc =\mathcal{C}(\mathfrak{sp}_{2n})_k$
and $\Dd =\mathcal{C}(\mathfrak{sp}_{2k})_n\boxtimes\mathcal{C}(\mathfrak{so}_{4nk})_1^{rev}$.
We have $\tilde A\cap (\Cc \boxtimes {\bf 1})={\bf 1}\boxtimes {\bf 1}\boxtimes \aLam_0$ and
$\tilde A\cap ({\bf 1}\boxtimes \Dd)={\bf 1}\boxtimes {\bf 1}\boxtimes \aLam_0\oplus {\bf 1}\boxtimes {\bf 1}^c\boxtimes \aLam_+=:B$ (note that the second summand is invertible of order 2). Thus by Example \ref{good equiv} we have a braid reversing equivalence
$\Cc \simeq \Dd_B^{dys}$ sending $\alam$ to the $B-$module $\alam^t\boxtimes \aLam_0\oplus
\alam^c\boxtimes \aLam_+=(\alam^t\boxtimes \aLam_0)\otimes B$ if $|\lambda|$ is even and to $\alam^t\boxtimes \aLam_1\oplus
\alam^c\boxtimes \aLam_-=(\alam^t\boxtimes \aLam_1)\otimes B$ if $|\lambda|$ is odd. 

Now consider the additive subcategory $\Dd_0$ of $\Dd =\mathcal{C}(\mathfrak{sp}_{2k})_n\boxtimes\mathcal{C}(\mathfrak{so}_{4nk})_1^{rev}$ generated by the objects
$\alam \boxtimes \aLam_0$ with $\lambda \in I_{k,n}^0$ and $\alam \boxtimes \aLam_1$ with $\lambda \in I_{k,n}^1$. It is clear that $\Dd_0$ is fusion subcategory of $\Dd$; moreover since 
the object $\aLam_1$ is a fermion (see Section \ref{tcat}), the category $\Dd_0$ identifies with
the category $\mathcal{C}(\mathfrak{sp}_{2k})_n^{-}$. On the other hand, the free module functor
$X\mapsto X\otimes B$ gives a braided equivalence $\Dd_0\simeq \Dd_B^{dys}$. Thus we get
a chain of equivalences 
$$\mathcal{C}(\mathfrak{sp}_{2n})_k=\Cc \simeq \Dd_B^{dys}\simeq \Dd_0=\mathcal{C}(\mathfrak{sp}_{2k})_n^{-}$$
sending $\alam$ to $\alam^t$. This completes the proof.

\end{proof}

\begin{rem} The proof above gives no information on uniqueness of the structure of braided tensor functor on the functor sending $\alam$ to $\alam^t$.
\end{rem}

\subsection{Second Proof}
In \cite{TW} braided fusion categories with the same fusion rules as $\mathcal{C}(\ssp_{2n},2n+2k+2)$ and $\Cc(\ssp_{2n},2n+2k+1)$ are reconstructed from their Grothendieck (semi-)rings and the eigenvalues of the braiding $c_{X,X}$ on the generating object $X$ analogous to the $2n$-dimensional representation of $\ssp_{2n}$, i.e. with highest weight $(1,0,\ldots,0)$. In particular they prove the following (cf. \cite[Section 9.3]{TW}):
\begin{thm}\label{twtheorem}
Let $\mathcal{C}$ and $\widetilde{\mathcal{C}}$ be braided fusion categories with  Grothendieck rings isomorphic to that of $\Cc(\ssp_{2n},\ell)$. Denote by $X$, $\tilde{X}$ the generating objects of $\Cc,\tilde{\Cc}$ corresponding to $X_{(1)}\in\Cc(\ssp_{2n},\ell)$, and $\be,\tilde{\be}$, $Y,\tilde{Y}$ and $W,\tilde{W}$ the objects corresponding to $X_{(0)},X_{(2)}$ and $X_{(1,1)}$ in $\Cc(\ssp_{2n},\ell)$.  If the eigenvalues of $c_{X,X}$ and $c_{\tilde{X},\tilde{X}}$ coincide on $[\be,Y,W]$ and $[\tilde{\be},\tilde{Y},\tilde{W}]$ then $\Cc$ and $\widetilde{\Cc}$ are equivalent as braided fusion categories.
\end{thm}

We remark that if $\ell$ is even we are in the symplectic case, while for $\ell$ odd this is called the ortho-symplectic case.  Notice also that in this subsection we denote the weights by the corresponding Young diagrams, i.e. $\lambda$ etc. rather than as $\widehat{\lambda}$ as is customary in the quantum group approach.

We now provide an alternate proof to Theorem \ref{main theorem} using Theorem \ref{twtheorem}.  

\begin{proof} 
Let $\ell=2k+2n+2$, and denote by $X_\lambda$ the simple objects in $\Cc(\ssp_{2n},\ell)$ and by $\tilde{X}_\mu$ the simple objects in $\Cc(\ssp_{2k},\ell)$, where $\lambda$ is a Young diagram with at most $n$ rows and at most $k$ columns and $\mu$ is a Young diagram with at most $k$ rows and at most $n$ columns.  Notice that $\Cc(\ssp_{2k},\ell)^{-rev}$ and $\Cc(\ssp_{2k},\ell)$ are identical as fusion categories.

We first must verify that $\Cc(\ssp_{2k},\ell)$ has the same fusion rules as $\Cc(\ssp_{2n},\ell)$.  Clearly there is a bijection between these label sets given by transpose.  To see that transpose provides an isomorphism of Grothendieck rings we note by \cite[Proposition 8.6]{TW} that the fusion rules for $\Cc(\ssp_{2n},\ell)$ and $\Cc(\ssp_{2k},\ell)$ are determined by the rule for tensoring with $X:=X_{(1)}$ (resp. $\tilde{X}:=\tilde{X}_{(1)}$.  These are as follows: $X_{(1)}\otimes X_{\lambda}$ is the direct sum of all $X_{\mu}$ whose Young diagram has one box more or one box less than $\hat{\lambda}$ with the same rule for $\tilde{X}$.  Clearly this rule is preserved when taking the transpose. Notice that under this isomorphism the object $X_{(2)}$ is mapped to $\tilde{X}_{(1,1)}$ and $X_{(1,1)}$ is mapped to $\tilde{X}_{(2)}$.

Next we must compare the eigenvalues of the braiding isomorphisms $c_{X,X}$ and $c_{\tilde{X},\tilde{X}}$.

 Observe that $X\otimes X\cong \mathbf{1}\oplus X_{(1,1)}\oplus X_{(2)}$.  The eigenvalues of $c_{X,X}$ on $X_\lambda\subset X^{\otimes 2}$ are computed as in \cite{Resh}: $\pm q^{\frac{1}{2}c_\lambda-c_{(1)}}$ where $c_\lambda:=\langle \lambda+2\rho,\lambda\rangle$ and where we take $+$ if $X_{\lambda}$ appears in $S^2X$ as representations of $U_q\ssp_{2n}$ and $-$ otherwise.  In this case $X_{(2)}\subset S^2 X$ while $\be=X_{(0)},X_{(1,1)}\subset \wedge^2 X$.  It follows that the eigenvalues for $c_{X,X}$ on $[\be,X_{(2)},X_{(1,1)}]$ are $[-q^{-2n-1},q,-q^{-1}]$, whereas for  $c_{\tilde{X},\tilde{X}}$ on $[\be,\tilde{X}_{(2)},\tilde{X}_{(1,1)}]$ are $[-q^{-2k-1},q,-q^{-1}]$.  The effect of transposing diagrams, followed by reversing the braiding and then an overall sign change takes $[-q^{-2k-1},q,-q^{-1}]$ to $[q^{2k+1},q,-q^{-1}]$, and since $q^{2k+1}=-q^{-2n-1}$ we have verified the eigenvalues match.

\end{proof}

In \cite{RowellMathZ} it is shown that the non-unitary ribbon category $\Cc(\ssp_{2n},2n+2k+1)$ has the same fusion rules as $\Cc(\sso_{2k+1},2n+2k+1)$.  We will prove a more precise result, but first we need some additional notation.  One may construct ribbon fusion categories from $U_q \sso_{2k+1}$ and $U_q \ssp_{2n}$ for any choice of $q$ with $q^2$ a primitive $\ell$th root of unity.  Since $\ell=2k+2n+1$ is odd there are, up to Galois conjugation, two choices of this parameter: $e^{\pi i/\ell}$ or $e^{\pi i/\ell}$. For ease of notation we will define $q:=e^{\pi i/\ell}$, and emphasize this parameter dependence by denoting the categories by $\Cc(\ssp_{2n},Q,\ell)$ and $\Cc(\sso_{2k+1},Q,\ell)$, where $Q^2$ is a primitive $\ell$th root of unity.

We first establish the following:
\begin{lem}\label{bclemma} $\Cc(\ssp_{2n},Q,\ell)$ and $\Cc(\sso_{2k+1},Q,\ell)$ are $\Z/2$-graded, with modular trivially graded component, for any choice of $Q$.  In fact we have $\Cc(\ssp_{2n},Q,\ell)\cong\Cc(\ssp_{2n},Q,\ell)_0\boxtimes \Cc(\Z/2,P_1)$ and $\Cc(\sso_{2n},Q,\ell)\cong\Cc(\sso_{2n},Q,\ell)_0\boxtimes \Cc(\Z/2,P_2)$ where $\Cc(\Z/2,P_i)$ are pointed ribbon categories associated with the pre-metric group $(\Z/2,P_i)$.
\end{lem}

\begin{proof} $\Cc(\ssp_{2n},Q,\ell)$ and $\Cc(\sso_{2k+1},Q,\ell)$ each have one non-trivial invertible object (see \cite{RowellMathZ}), labeled by the weight $(\ell-2n,0,\ldots,0)$ for type $C$ and $\frac{1}{2}(\ell-2k,\ldots,\ell-2k)$ for type $B$.  Let us denote these by $\eta$ for type $C$ and $\gamma$ for type $B$.  

Computing the twists and dimension we find that $\eta$ is a non-unitary fermion ($\dim(\eta)=-1$, $\theta_\eta=1)$ if $Q^\ell=1$ and a non-unitary boson $(\dim(\eta)=-1$, $\theta_\eta=-1)$ if $Q^\ell=-1$.  Moreover, we compute $$c_{\eta,X_{(1)}}c_{X_{(1)},\eta}=\frac{\theta_\eta\theta_{X_{(1)}}}{\theta_{\eta\otimes X_{(1)}}}Id_{X_{(1)}\otimes\eta}=Id_{X_{(1)}\otimes\eta}$$ so that $\eta$ is transparent regardless of the value of $Q^\ell$, and in fact generates the M\"uger center of $\Cc(\ssp_{2n},Q,\ell)$.  In particular the subcategory with simple objects labeled by Young diagrams with an even number of boxes $\Cc(\ssp_{2n},Q,\ell)_0$ does not contain $\eta$ but has centralizer $\langle\eta\rangle$, giving the desired factorization and modularity $\Cc(\ssp_{2n},Q,\ell)_0$.

Similarly we compute that $\gamma$ is a transparent boson ($\dim(\gamma)=\theta_\gamma=\pm 1)$ or fermion ($\dim(\gamma)=-\theta_\gamma=\pm 1$) if $k$ is even or if $k$ is odd
 and $Q^\ell=1$ and a semion ($\dim(\gamma)=\pm 1$ and $\theta_\gamma=\pm i$) otherwise. In either case this shows that the subcategory $\Cc(\sso_{2k+1},Q,\ell)_0$ with simple objects labeled by integer weights forms a modular subcategory, and we have the desired factorization.

\end{proof}

We can now prove the following:
\begin{thm} \label{nonunitary}
Let $\ell=2k+2n+1$ and $q=e^{\pi i/\ell}$. Then there are braid-reversing equivalences between
\begin{enumerate}
\item $\Cc(\ssp_{2n},q,\ell)_0$ and $\Cc(\sso_{2k+1},q^{\frac{\ell+1}{2}},\ell)_0$  and
\item $\Cc(\ssp_{2n},q^2,\ell)_0$ and $\Cc(\sso_{2k+1},q,\ell)_0$.
\end{enumerate}
\end{thm}

\begin{proof}
We again employ Theorem \ref{twtheorem}.  The Grothendieck rings of $\Cc(\sso_{2k+1},\ell)$ and $\Cc(\ssp_{2n},\ell)$ were already established to be isomorphic in \cite{RowellMathZ}, denote that isomorphism by $\Phi$. To avoid confusion we will denote simple objects in $\Cc(\sso_{2k+1},Q,\ell)$ by $Y_{\lambda}$ and those in $\Cc(\ssp_{2n},P,\ell)$ by $X_{\mu}$. Under the isomorphism $\Phi$ we have that $\Phi(Y_{(1)})=X_{(1)}\otimes X_{\eta}=X_{(\ell-2n-1,0,\ldots,0)}=X^\prime$, and these objects generate the respective modular subcategories $\Cc(\sso_{2k+1},Q,\ell)_0$ and $\Cc(\ssp_{2n},P,\ell)_0$ as in Lemma \ref{bclemma}.  Notice that it is enough to identify the eigenvalues of the braidings on $Y_{(1)}^{\otimes 2}$ and $(X^\prime)^{\otimes 2}$: we may lift this identification to $\Cc(\sso_{2k+1},Q,\ell)$ and $\Cc(\ssp_{2n},P,\ell)$ by tensoring with pointed categories $\Cc(\Z/2,P_i)$ and then apply Theorem \ref{twtheorem}.

The eigenvalues of $c_{Y_{(1)},Y_{(1)}}$ on $[\be,Y_{(2)},Y_{(1,1)}]$ for $\Cc(\sso_{2k+1},Q,\ell)$ are found in \cite[Table 6.1]{LRW}, the are $[Q^{-4k},Q^2,-Q^{-2}]$.  We have $\Phi(Y_{(2)})=X_{(1,1)}$ and vice versa, so transpose followed by braid-reversing gives us the eigenvalues $[Q^{4k},-Q^2,Q^{-2}]$.

For $\Cc(\ssp_{2n},P,\ell)$ the eigenvalues of $c_{X_{(1)},X_{(1)}}$ on $[\be,X_{(2)},X_{(1,1)}]$ are $[-P^{-2n-1},P,-P^{-1}]$ \cite[Table 6.1]{LRW}.  If $P^\ell=-1$ then $\eta$ is a boson so the braiding eigenvalues on $X^\prime$ and $X_{(1)}$ are identical.  Thus if $P=q=e^{\pi i/\ell}$ we must take $Q=q^{\frac{\ell+1}{2}}$ so that $Q^{4k}=q^{2k(\ell+1)}=q^{2k}=-q^{-2n-1}$, and $-Q^2=-(q^{\ell+1})=q$ and $Q^{-2}=-q^{-1}$ so that the eigenvalues match.  If $P^\ell=1$ then $\eta$ is a fermion so that the braiding eigenvalues on $X^\prime$ differ from those on $X_{(1)}$ by an overall sign, giving: $[P^{-2n-1},-P,P^{-1}]$.  Thus if we take $P=q^2=e^{2\pi i/\ell}$ we choose $Q=q$ so that $Q^{4k}=P^{2k}=P^{-2n-1}$, $-Q^2=-P$ and $Q^{-2}=P^{-1}$.  

\end{proof}

We close this section with some remarks on the advantages of this categorical approach.
\begin{rem}
\begin{enumerate}
\item The non-unitary categories $\Cc(\ssp_{2n},2n+2k+1)$ and $\Cc(\sso_{2k+1},2n+2k+1)$ cannot be constructed from affine Lie algebras, to our knowledge.  Twisted affine Lie algebras at fractional levels (see e.g. \cite{Kac}) provide similar combinatorics, but there is no level-preserving fusion product.
\item The results of \cite{TW} provide a description of the categories $\Cc(\ssp_{2n},l)$ via generators and relations. One deduces that the functor between $\Cc(\ssp_{2k},2k+2n+2)$ and $\Cc(\ssp_{2n},2k+2n+2)$ sending $X$ to $\tilde X$ has a \emph{unique} braided tensor structure up to
isomorphism of tensor functors, see \cite{EM}. Similar uniqueness holds for functors from Theorem \ref{nonunitary} provided that the functor sends $X$ to $X^\prime$.
\end{enumerate}
\end{rem}

\section{Appendix: Kac-Peterson formula in the symplectic case} \label{KPproof}
\subsection{}
The main goal of this Section is to give a proof of Proposition \ref{KPcc} based on
\cite[Proposition 1]{KP}. 

Let us recall the setup. Let $\mathfrak{g}$ be a simple finite dimensional Lie algebra and
let $\theta$ be an automorphism of $\mathfrak{g}$ of order 2. Let $\mathfrak{t}=\{ x\in \mathfrak{g}|
\theta(x)=x\}$ be the subspace of $\theta-$invariant vectors and let $\mathfrak{p}=\{ x\in \mathfrak{g}| \theta(x)=-x\}$. The following result is standard:

\begin{lem}[Cartan decomposition]\label{cartandecomp} We have a decomposition $\mathfrak{g} = \mathfrak{t}\oplus\mathfrak{p}$ satisfying the following 
\begin{itemize}
\item[(i)] $\mathfrak{t}$ is a reductive Lie subalgebra of $\mathfrak{g}$ and $[\mathfrak{t}, \mathfrak{p}] \subset \mathfrak{p}$.
\item[(ii)] The restricted Killing form $\kappa|_{\mathfrak{p}\times\mathfrak{p}}$ is a non-degenerate symmetric bilinear form preserved by the action of $\mathfrak{t}$.
\item[(iii)] The action of $\mathfrak{t}$ on $\mathfrak{p}$ gives an embedding $\mathfrak{t}\hookrightarrow\mathfrak{so}(\mathfrak{p}, \kappa|_{\mathfrak{p}\times\mathfrak{p}})\cong \mathfrak{so}_{\dim\mathfrak{p}}$.
\end{itemize}
\end{lem}

We will be interested in the following special case.

\begin{eg}\label{tpg} Let $\C^{2n+2k}=\C^{2n}\oplus \C^{2k}$ be a direct sum of two symplectic 
spaces. Let $\theta \in GL(\C^{2n+2k})$ be the linear operator acting by -1 on $\C^{2n}$ and
by 1 on $\C^{2k}$. It is clear that $\theta$ preserves the symplectic form, so it acts by
conjugations on $\mathfrak{g}= \mathfrak{sp}(\C^{2n+2k})=\mathfrak{sp}_{2n+2k}$.

Then $\mathfrak{t}\cong \mathfrak{sp}_{2n}\oplus\mathfrak{sp}_{2k}$ and $\mathfrak{p}$ must have dimension $4nk$ so we get an embedding $\mathfrak{sp}_{2n}\oplus\mathfrak{sp}_{2k}\subset \mathfrak{so}_{4nk}$ from Section \ref{ortho}.
\end{eg}

This example enjoys the following extra property: there exists a Cartan subalgebra $\mathfrak{h}$
of $\mathfrak{g}$ which is contained in $\mathfrak{t}$. We choose and fix such a subalgebra.
Let $\Delta$ and $\Delta_{\mathfrak{t}}$ denote the root systems of $\mathfrak{g}$ and $\mathfrak{t}$ with respect to $\mathfrak{h}$. Let $W$ and $W_{\mathfrak{t}}$ be the corresponding
Weyl groups. We choose a set $\Delta^+$ of positive roots for $\Delta$; then $\Delta^+_{\mathfrak{t}}=\Delta_{\mathfrak{t}}\cap \Delta$ is a set of positive roots for $\Delta_{\mathfrak{t}}$. Let $\rho$ and $\rho_{\mathfrak{t}}$ be the sums of fundamental weights.
Finally let $\hat W, \hat W_{\mathfrak{t}}, \hat \Delta, \hat \Delta_{\mathfrak{t}}$ etc denote the
affine versions of the notions above.

The set
$$W_{\mathfrak{t}}^1=\{w\in W\,|\, \Delta_{\mathfrak{t}}^+\subset w\Delta^+\}=\{w\in W\,|\, w^{-1}\Delta_{\mathfrak{t}}^+\subset \Delta^+\}$$
and its affine counterpart
$$\hat{W}_{\mathfrak{t}}^1=\{w\in \hat{W}\,|\,\hat{\Delta}_{t}^+\subset w\hat{\Delta}^+\}=\{w\in \hat{W}\,|\,w^{-1}\hat{\Delta}_{t}^+\subset \hat{\Delta}^+\}$$
will play a significant role in what follows in view of the following result:

\begin{thm}\label{KP} Let $\mathfrak{g}$, $\mathfrak{t}$ and $\mathfrak{h}$ be as above. 

{\em (i) (Lemma 2.2 of \cite{Part})} Let $S$ be the spinor representation of $\mathfrak{so}_{\dim\mathfrak{p}}$ restricted to $\mathfrak{t}$ and let $L(\mu)$ be the irreducible $\mathfrak{t}-$module of highest weight $\mu$.. Then there is an isomorphism of $\mathfrak{t}-$modules
$$S\cong\bigoplus_{w\in {W}^1} L(w({\rho})-{\rho_{\mathfrak{t}}}).$$

{\em (ii) (Proposition 1 of \cite{KP})}
The affinization of the embedding $\mathfrak{t}\subset \mathfrak{so}_{\dim\mathfrak{p}}$ is a conformal embedding $\widehat{\mathfrak{t}}\subset (\widehat{\mathfrak{so}}_{\dim\mathfrak{p}})_1$. Let $\hat S$ be the spinor representation
of $\widehat{\mathfrak{so}}_{\dim\mathfrak{p}}$ and let $\hat L(\mu)$ be the irreducible $\widehat{\mathfrak{t}}-$module of highest weight $\mu$.
Then there is an isomorphism of $\widehat{\mathfrak{t}}$-modules
$$\hat S\cong\bigoplus_{w\in \hat{W}^1} \hat L(w(\hat{\rho})-\hat{\rho_{\mathfrak{t}}}).$$
\end{thm}

\begin{rem} 
(i) The levels of affine Lie algebras appearing in Theorem \ref{KP} (ii) can be computed as follows:
Let $\mathfrak{t}_1$ be a direct summand of the Lie algebra $\mathfrak{t}$. Then the level
of $\widehat{\mathfrak{t}}_1$ is the ratio of the Killing form of $\mathfrak{g}$ restricted to
$\mathfrak{t}_1$ and the Killing form of $\mathfrak{t}_1$. In particular in the setup of Example
\ref{tpg} we get the conformal embedding \eqref{tg}.

(ii) Note that the dimension of space $\mathfrak{p}$ is even (this is twice the cardinality of
$\Delta^+ \setminus \Delta^+_{\mathfrak{t}}$). Thus both $S$ and $\hat S$ are the sums of two half-spinor modules. In particular in the setup of Example \ref{tpg} we have $\hat S=\aLam_+\oplus \aLam_-$.
\end{rem}
\subsection{The symplectic case}
We will make Theorem \ref{KP} explicit in the setup of Example \ref{tpg}.
Let $$\{e_1-e_2,e_2-e_3,\dots,e_{n+k-1}-e_{n+k},2e_{n+k}\}$$ be the simple roots of $\mathfrak{sp}_{2n+2k}$. Then the simple roots of $\mathfrak{sp}_{2n}$ are
$$\{e_1-e_2,e_2-e_3,\dots,e_{n-1}-e_n,2e_n\}$$and the simple roots of $\mathfrak{sp}_{2k}$ are
$$\{e_{n+1}-e_{n+2},e_{n+2}-e_{n+3},\dots,e_{n+k-1}-e_{n+k},2e_{n+k}\}.$$

 Let $\Delta_{\mathrm{long}}^+=\{ 2e_1, \ldots , 2e_{n+k}\}$ be the set of long positive roots of $\mathfrak{sp}_{2n+2k}$. 
 Since the long positive roots of $\mathfrak{sp}_{2n}$ are $2e_{1},\dots , 2e_{n}$ and of $\mathfrak{sp}_{2k}$ are $2e_{n+1},\dots,2e_{n+k}$, we have $\Delta_{\mathrm{long}}^+\subset \Delta^+_{\mathfrak{t}}$.

\begin{prop}\label{affine}  In the setup of Example \ref{tpg}, $\hat{W}_{\mathfrak{t}}^1=W_{\mathfrak{t}}^1$. 
\end{prop}
\begin{proof} Let $w'\in\hat{W}_{\mathfrak{t}}^1$. Recall $\hat{W}=W\rtimes T$ (\cite[Proposition 6.5]{Kac}) where
$T$ is the group of translations. So $w'=wt$ for some $w\in W$ and $t^{-1}=t_{m_1e_1+\dots +m_{n+k}e_{n+k}}\in T$ for $m_i\in\Z$. Let $\delta\in \hat{\Delta}_{t}^+$ be the indivisible imaginary root. Consider

$$\begin{aligned}t^{-1}w^{-1}(w(\delta-2e_i))&=t^{-1}(\delta)-t^{-1}(2e_i)\\
                        &=\delta-2e_i+2m_i\delta.\end{aligned}$$
                       
Since $w(\delta-2e_i)=\delta-w(2e_i)\in\Delta^+_{\mathfrak{t}}$ and $w'\in\hat{W}_{\mathfrak{t}}^1$, it follows $m_i\geq0$, see Section \ref{affroots}.

Similarly $w(\delta+2e_i)\in\Delta^+_{\mathfrak{t}}$ and $(w')^{-1}w(\delta+2e_i)=\delta+2e_i-2m_i\delta$. Thus $w'\in\hat{W}_{\mathfrak{t}}^1$ implies $m_i\leq0$. Hence $m_i=0$.
\end{proof}

\begin{rem} In view of Theorem \ref{KP} the Proposition \ref{affine} says that the branching rules
for the finite dimensional and affine cases are ``the same'', cf Remark \ref{ccfinite} (ii).
\end{rem}

Let us describe $W_{\mathfrak{t}}^1$. Recall that the group $W$ is the group of signed permutations; let $S_{n+k}\subset W$ be the subgroup of permutations without signs.
Recall that $\Delta_{\mathrm{long}}^+\subset \Delta^+_{\mathfrak{t}}$. Thus we have
$$W^1=\{w\in W \,|\, w^{-1}\Delta_{\mathrm{long}}^+=\Delta_{\mathrm{long}}^+, w^{-1}(e_i-e_{i+1}) \in \Delta^+, \text{for all }i\neq n\}.$$
Observe that the the first condition implies that $w^{-1}\in S_{n+k}$. Thus we have

\begin{lem}\label{permute} The set $W_{\mathfrak{t}}^1$ is contained in $S_{n+k}\subset W$. 
A permutation $s\in S_{n+k}$ is in $W_{\mathfrak{t}}^1$ if and only if
$$s^{-1}(1)<s^{-1}(2)<\dots <s^{-1}(n), s^{-1}(n+1)<s^{-1}(n+2)<\dots<s^{-1}(n+k).\;\; \square$$
\end{lem}

Now if we consider $n+k$ dots on a straight line and paint the dots numbered $s^{-1}(1), s^{-1}(2), \dots ,s^{-1}(n)$ in black and the dots numbered $s^{-1}(n+1), s^{-1}(n+2), \dots ,s^{-1}(n+k)$ in white we get precisely ``black and white dots diagram'' as in Example \ref{bwdots}.
Conversely from such a diagram we get a unique permutation as in Lemma \ref{permute}.
Thus we constructed a bijection between $W_{\mathfrak{t}}^1$ and the set $C_{n,k}=I_{n,k}$ from
Section \ref{combinatorics}. 

\begin{eg} The permutation $s^{-1}$ corresponding to the diagram
$$\bullet\bullet\circ\bullet\bullet\bullet\bullet\circ\circ\circ\bullet\circ \circ$$
from example \ref{bwdots} is 
$$\left( \begin{array}{ccccccccccccc}1&2&3&4&5&6&7&8&9&10&11&12&13\\
1&2&4&5&6&7&11&3&8&9&10&12&13\end{array}\right) .$$
\end{eg}

It remains to compute the weights $w({\rho})-{\rho_{\mathfrak{t}}}$ for $w\in W_{\mathfrak{t}}^1$
(we can restrict ourselves to the finite case in view of Proposition \ref{affine}). We have
$$\rho =(n+k)e_1+(n+k-1)e_2+\ldots +2e_{n+k-1}+e_{n+k}=\sum_{i=1}^{n+k}(n+k+1-i)e_i,$$
and
$$\rho_{\mathfrak{t}}=ne_1+(n-1)e_2+\ldots +e_n+ke_{n+1}+(k-1)e_{n+2}\ldots +e_{n+k}.$$
For a permutation $s\in S_{n+k}$ we have
$$s\rho =\sum_{i=1}^{n+k}(n+k+1-i)e_{s(i)}=\sum_{i=1}^{n+k}(n+k+1-s^{-1}(i))e_i.$$

Thus we have
$$s\rho -\rho_{\mathfrak{t}}=\sum_{i=1}^{n}(k+i-s^{-1}(i))e_i+\sum_{i=1}^{k}(n+i-s^{-1}(i))e_{n+i}.$$
Here the first summand represents the weight of $\mathfrak{sp}_{2n}\subset \mathfrak{sp}_{2n}\oplus\mathfrak{sp}_{2k}=\mathfrak{t}$ and the second summand represents the weight of $\mathfrak{sp}_{2k}\subset \mathfrak{sp}_{2n}\oplus\mathfrak{sp}_{2k}=\mathfrak{t}$. It is clear
that the second weight is computed similarly to the first one with replacement of the sequence
$s^{-1}(1)<s^{-1}(2)<\dots <s^{-1}(n)$ by the sequence $s^{-1}(n+1)<s^{-1}(n+2)<\dots<s^{-1}(n+k)$ or, equivalently, by replacing all the black dots by the white ones and vice versa. This is precisely the description of the bijection $\lambda \mapsto \lambda^c$ in the language of diagrams. 
Thus Proposition \ref{KPcc} is proved.

 \end{document}